\documentclass[final,1p,times]{elsarticle}

\usepackage{epsfig,amsmath,amssymb,amsthm,color,amsfonts,epstopdf,mathrsfs}

\everymath{\displaystyle}

\def\CC{{\mathbb{C}}}

\def\RR{{\mathbb{R}}}
\def\NN{{\mathbb{N}}}

\newcommand{\bb}[1]{\begin{equation}\label{#1}}
\newcommand{\ee}{\end{equation}}
\newcommand{\bbb}{\begin{eqnarray}}
\newcommand{\eee}{\end{eqnarray}}
\newcommand{\bbbb}{\begin{eqnarray*}}
\newcommand{\eeee}{\end{eqnarray*}}

\newcommand{\nnn}{\nonumber}

\definecolor{green1}{rgb}{0.1,0.5,0.0}

\newtheorem{theorem}{Theorem}
\newtheorem{lemma}{Lemma}

\theoremstyle{remark}
\newtheorem{rem}{Remark}
\theoremstyle{definition}
\newtheorem{definition}{Definition}
\theoremstyle{plain}
\newtheorem{assume}{Assumption}
\theoremstyle{definition}
\newtheorem{example}{Example}

\usepackage[usenames,dvipsnames,svgnames,table]{xcolor}

\newcommand{\clearallnum}{
    \numberwithin{equation}{section} 
    \numberwithin{theorem}{section} 
    \numberwithin{lemma}{section} 
    \numberwithin{cor}{section} 
    \numberwithin{rem}{section} 
    \numberwithin{definition}{section} 
    \numberwithin{example}{section}
    \numberwithin{figure}{section}
    }

\journal{}

\begin{document}

\begin{frontmatter}



\title{Analysis of an approximation to a fractional extension problem}


\author{Joshua L. Padgett}
\ead{joshua.padgett@ttu.edu}

\address{Department of Mathematics and Statistics, Texas Tech University, Broadway and Boston, Lubbock, TX 79409-1042}

\begin{abstract}
The purpose of this article is to study an approximation to an abstract Bessel-type problem, which is a generalization of the extension problem associated with fractional powers of the Laplace operator. Motivated by the success of such approaches in the analysis of time-stepping methods for abstract Cauchy problems, we adopt a similar framework herein. The proposed method differs from many standard techniques, as we approximate the true solution to the abstract problem, rather than solve an associated discrete problem. The numerical method is shown to be consistent, stable, and convergent in an appropriate Banach space. These results are built upon well understood results from semigroup theory. Numerical experiments are provided to demonstrate the theoretical results.
\end{abstract}


\end{frontmatter}



\section{Introduction} \clearallnum

Due to its wide array of applications, the computation of fractional powers of the Laplacian, and other elliptic operators, has become a problem of great interest \cite{antil2015fem,chen2015pde,nochetto2015pde,nochetto2016pde}. Fractional powers of operators have received this attention due to their applicability to the accurate modeling of real-world problems with varying scales. Examples of such problems are found in porous media flow, peridynamics, nonlocal continuum field theory, finance, and many others (see \cite{10.1086/338705,cushman1993nonlocal,eringen2002nonlocal,silling2000reformulation}, and the references therein). However, due to the nonlocal nature of fractional problems, the development of accurate and efficient computational algorithms has been greatly hindered. In response to this hurdle, the revolutionary work by Caffarelli and Silvestre demonstrated that the nonlocal fractional Laplacian problem may be recast into an equivalent local problem which is amenable to certain standard numerical techniques \cite{doi:10.1080/03605300600987306}.

In order to demonstrate, fix $0<s<1$ and consider the function $v\,:\,\mathbb{R}^d \to \mathbb{R}$ which solves
\bb{aaa1}
(-\Delta)^s v = f, \quad x\in\mathbb{R}^d,\\
\ee
where $f$ is a given function of appropriate regularity and the fractional power of the Laplacian is defined via the hypersingular integral
\bb{fracdef}
(-\Delta)^sv(x) \mathrel{\mathop:}= C_{d,s}\int_{\RR^d}\frac{v(x) - v(z)}{|x-z|^{d+2s}}\,dz
\ee
with $C_{d,s}$ being some normalization constant \cite{doi:10.1080/03605300600987306}. Caffarelli and Silvestre then showed that for $0<s<1,$ the fractional Laplacian in $(\ref{aaa1})$ can be realized via the Dirichlet-to-Neumann map for a function $u\,:\,\RR^d\times\RR_+\to\RR.$ The function $u$ then satisfies the following Bessel-type equation
\bb{e1}
\left\{\begin{array}{ll}
\partial_t^2 u(x,t) + \frac{1-2s}{t}\partial_t u(x,t) = -\Delta_x u(x,t), & \quad (x,t)\in \RR^d\times \RR_+,\\
u(x,0) = v(x), & \quad x\in \RR^d,
\end{array}\right.
\ee
where $v$ is the solution to $(\ref{aaa1}).$ Moreover, the function $u$ satisfying $(\ref{e1})$ may be viewed as the harmonic extension of $v$ to the fractional dimension $2-2s.$ One can then calculate $(-\Delta)^sv$ as
\bb{e2}
c_sf(x) = c_s(-\Delta)^sv(x) = -\lim_{t\to 0^+} t^{1-2s}\partial_t u(x,t),\quad x\in \RR^d,
\ee
where $c_s$ is a constant depending upon $s,$ but not on the dimension $d,$ and $u$ solves $(\ref{e1}).$ Moreover, the final equality is $(\ref{e2})$ is known as the {\em conormal derivative}. We have emphasized the relationship between $v$ and $f$ in $(\ref{e2})$ in order to clarify the ensuing exposition. Further details and applications of this method may be found in \cite{doi:10.1080/03605302.2017.1363229,doi:10.1080/03605300600987306,Gale2013,doi:10.1080/03605301003735680}.

Several numerical algorithms for solving problems involving fractional powers of the Laplacian have been developed based upon the Caffarelli-Silvestre extension technique \cite{d2013fractional,huang2014numerical,nochetto2015pde,nochetto2016pde}. These methods have also been extended to problems involving general elliptic operators and various other physically relevant nonlocal problems. Despite the increased interest in such methods, these explorations are still in their infancy and merit further consideration. Moreover, there is the need to improve the efficiency and accuracy of existing numerical methods.

In order to develop more robust approximations to $(\ref{e1}),$ we consider the following abstract Bessel-type problem
\bb{e3}
\left\{\begin{array}{ll}
u''(t) + \frac{\alpha}{t}u'(t) = -Au(t), &\quad t\in(0,\infty),\\
u(0) = u_0 \in X,
\end{array}\right.
\ee
where $X$ is a real separable Banach space, $A\,:\,D(A)\subseteq X\to X$ is a densely defined linear operator, and $\alpha\mathrel{\mathop:}= 1-2s,\ s\in (0,1).$ There have been numerous studies concerned with approximating particular forms of $(\ref{e3}),$ but herein we present a novel method and demonstrate its properties in an arbitrary Banach space for appropriate abstract operators $A.$ In particular, such analysis allows for {\em mesh independent} results, which are quite useful in application.

Analysis and representations of the true solution of $(\ref{e3})$ have been studied in multiple settings \cite{doi:10.1080/03605302.2017.1363229,Gale2013,doi:10.1080/03605301003735680}. However, to the author's knowledge, approximations to $(\ref{e3})$ in this abstract setting have yet to be explored. In particular, the true solution to $(\ref{e3})$ involves the abstract semigroup generated by the operator $A$ and there have been few methods employing direct approximations of this operator. These approximations preserve numerous desirable properties inherent to the original semigroup operator, often resulting in superior approximations, resulting in qualitatively superior approximations \cite{budd1999geometric,hairer2006geometric,Josh1,PADGETT2018210}.

This article is organized as follows. In Section 2 relevant mathematical preliminaries are outlined for the reader's convenience. Section 3 is concerned with the development and stability of the proposed approximation to $(\ref{e1}).$ Section 4 provides the necessary results to guarantee convergence of the proposed method. Section 5 presents some concrete example of our abstract problem and provides numerical experiments verifying the analytic result. Finally, concluding remarks and an outline of future endeavors are given in Section 6.

\section{Preliminaries} \clearallnum

Let $X$ be a real separable Banach space with norm denoted $\|\cdot\|$ and let $B(X)$ denote the Banach algebra of bounded operators on $(X,\|\cdot\|).$ We wish to consider the following Bessel-type problem
\bb{b1}
\left\{\begin{array}{ll}
u''(t) + \frac{\alpha}{t}u'(t) = -Au(t), & \quad t\in (0,\infty),\\
u(0) = u_0 \in X, & 
\end{array}\right.
\ee
where $\alpha\mathrel{\mathop:}=1-2s,\ s\in (0,1),$ and $A\,:\,D(A)\subseteq X\to X$ is linear. The operator norm on $X$ will also be denoted by $\|\cdot\|.$ Before providing details regarding a solution to $(\ref{b1})$ we detail a necessary assumption and some relevant background.

\begin{assume}
Let the operator $A$ be closed and densely defined in $X.$ Moreover, let $A$ be strictly $m-$dissipative in $X.$ 
\end{assume}

We recall that an operator $A$ is called {\em strictly dissipative} if and only if
$$\|(I-tA)u\| > \|u\|,\quad \mbox{for all}\ u\in D(A),\ t\ge 0,$$
and it is called {\em strictly $m-$dissipative} if and only if it satisfies the additional range condition
$$R(I-tA) = X,\quad\mbox{for all}\ t\ge 0.$$
The notion of a dissipative operator may also be formalized in the following manner. Define the (left) semi-inner-product $[\cdot,\cdot]\,:\,X\times X\to \RR$ on the Banach space $X$ as
\bb{semi-inner}
[u,v]\mathrel{\mathop:}= \|v\|\lim_{\varepsilon\to 0^-}\frac{1}{\varepsilon}\left(\|v + \varepsilon u\| - \|v\|\right),\quad u,v\in X.
\ee
Then $A$ is strictly dissipative if and only if
$$\mbox{Re}[Au,u] < 0,\quad\mbox{for all}\ u\in D(A).$$
Note that the limit in $(\ref{semi-inner})$ exists as $[\cdot,\cdot]$ is a G-differential of the norm $\|\cdot\|.$ If $X$ is a Hilbert space then we have that $(\ref{semi-inner})$ coincides with the true inner product on $X,$ that is $[\cdot,\cdot] = \langle \cdot,\cdot\rangle.$ Moreover, we have that $[u,u] = \|u\|^2,$ for any $u\in X,$ and
$$-\|u\|\,\|v\| \le [u,v]\le \|u\|\,\|v\|,$$
for any $u,v\in X.$ More details regarding properties of $(\ref{semi-inner})$ may be found in \cite{deimling2010nonlinear}.

\begin{rem}
It is worth noting that many existing works in the direction of abstract numerical analysis often require that an operator simply be {\em m-dissipative}. The current work requires a slightly stricter result in order to guarantee the appropriate decay of the solution (a requirement similar to that in \cite{Gale2013,nochetto2015pde}).
\end{rem}

From this, it readily follows that the resolvent operator $R(t,A)\mathrel{\mathop:}= (I-tA)^{-1}\,:\,X\to X$ of a strictly $m-$dissipative operator is well-defined and a contraction in $X.$ It also follows that strictly $m-$dissipative operators generate strongly continuous semigroups of contractions \cite{hansen2013convergence,lumer1961dissipative}. We will denote the semigroup generated by $A$ by $(T(t))_{t\ge 0}.$ Note that if $X$ is reflexive then every strictly $m-$dissipative operator is densely defined. More details regarding $m-$dissipative operators and their properties can be found in \cite{barbu2010nonlinear,lumer1961dissipative}.

The following example demonstrates that standard elliptic operators are examples of the proposed abstract formulation.

\begin{example}
Let $\Omega\subset \RR^d$ be a bounded domain with smooth boundary. We may then consider the operator $A = \Delta$ with homogeneous Dirichlet boundary conditions, on $\partial\Omega.$ Let $X = H_0^1(\Omega)\cap H^2(\Omega),$ and for $u,v\in X,$ we define
$$[u,v] = \int_\Omega uv\, d\Omega\quad \mbox{and}\quad \|u\|^2 = [u,u].$$
We then have, by employing integration by parts and standard Sobolev-type inequalities, that
$$[Au,u] = -[\nabla u,\nabla u] \le -C[u,u],$$
where $C$ is some positive constant. Thus, we have that $\mbox{Re}[Au,u] < 0$ for all $u\in X.$ The range condition follows from standard results, as well.
\end{example}

We also define the logarithmic norm of an operator $A$ with respect to the semi-inner-product $[\cdot,\cdot]$ given by $(\ref{semi-inner})$ to be 
\bb{lognorm}
\mu(A) \mathrel{\mathop:}= \sup_{u\in D(A)}\frac{[Au,u]}{\|u\|^2}.
\ee
Logarithmic norms have been well studied and more details regarding their properties may be found in \cite{soderlind2006logarithmic}. 

\begin{example}
If we consider the situation in Example 2.1, again, then it is easy to see that $\mu(A) < 0,$ where $A = \Delta.$
\end{example}

\begin{rem}
It is worth noting that one could easily have considered (the negation) of any strongly elliptic operator in Examples 2.1 and 2.2, not just the Laplacian. In fact, one could equivalently define strongly elliptic operators as any operator $L$ such that $\mu(L) > 0.$
\end{rem}

\begin{rem}
While not the focus of this work, one may easily use $(\ref{lognorm})$ to define the notion of ellipticity of nonlinear operators on $X.$ This approach can be quite useful and allows for the consideration of more generalized notions of numerical stability for stiff systems of equations.
\end{rem}

The above formulations allow us to develop the following, very useful, result.

\begin{lemma}
Let $A\,:\,D(A)\subseteq X \to X$ be the generator of a strongly continuous semigroup $(T(t))_{t\ge 0}$ in $D(A).$ Then, 
$$\|T(t)\| \le e^{t\mu(A)},\quad t\ge 0.$$
\end{lemma}

\begin{proof}
This result has been proven elsewhere, but we include a proof for completeness. Let $A$ be the generator of the semigroup $(T(t))_{t\ge 0}.$ Then it follows that $x(t) = T(t)x_0$ is the solution to
\bb{log1}
\left\{\begin{array}{ll}
x'(t) = Ax(t), & \quad t>0,\\
x(0) = x_0.
\end{array}\right.
\ee
Recall that the upper right Dini derivative of a function $x(t)$ is defined as
$$D_t^+x(t) \mathrel{\mathop:}= \limsup_{\varepsilon\to 0^+} \frac{x(t+\varepsilon) - x(t)}{\varepsilon}.$$
By direct calculation, we have
\bb{log2}
D_t^+\|x(t)\| = \frac{[Ax,x]}{\|x\|^2}\|x\|
\ee
(see, for instance, \cite{deimling2010nonlinear}). Employing $(\ref{lognorm})$ and solving $(\ref{log2})$ directly yields the desired result. 
\end{proof}

The remainder of this section will serve as a brief introduction to fractional powers of operators in Banach spaces and provide some existence results for problem $(\ref{b1})$. These introductory notes are based primarily on the reference \cite{martinez2001theory}, with some influence from \cite{meichsner2017fractional}.
Let $L(X)$ denote the set of all bounded linear operators mapping $X$ into $X$ and 
$$\rho(A)\mathrel{\mathop:}= \left\{\lambda\in\CC\,:\, (\lambda-A)\ \mbox{is injective and}\ (\lambda-A)^{-1}\in L(X)\right\}$$
denote the resolvent set of the linear operator $A.$ For operators $A\,:\,D(A)\subseteq X \to X$ satisfying Assumption 2.1 and $u\in D(A),$ we define the {\em Balakrishnan operator} with power $s\in(0,1)$ and base $-A$ as
\bb{bala}
J_{-A}^su \mathrel{\mathop:}= \frac{1}{\Gamma(s)\Gamma(1-s)}\int_0^\infty z^{s-1}(z-A)^{-1}(-A)u\,dz.
\ee
This operator will be used to define fractional powers of the operator $A.$ For more details see \cite{martinez2001theory}.

\begin{rem}
The operator $J_{-A}^s$ inherits numerous desirable properties from the underlying operator $A.$ For instance, if $A$ is bounded or injective, then so is $J_{-A}^s$ \cite{martinez2001theory}.
\end{rem}

With $(\ref{bala}),$ we are now able to define fractional powers of abstract operators.

\begin{definition}
Let $A$ be as in Assumption 2.1 and fix $s\in(0,1).$ Then we define $(-A)^s$ as
\begin{itemize}
\item[i.] $(-A)^s \mathrel{\mathop:}= J_{-A}^s,$ for $A\in L(X);$
\item[ii.] $(-A)^s \mathrel{\mathop:}= \left(J_{(-A)^{-1}}^s\right)^{-1},$ for $A$ being unbounded.
\item[iii.] $(-A)^s \mathrel{\mathop:}= \lim_{\delta \to 0^+} (-A + \delta)^s u,$ for $A$ being unbounded, $0\in\sigma(A)$ (where $\sigma(A)$ is the spectrum of $A$), and $D((-A)^s)$ given by
\bbb
&& D(A) = \left\{u\in\overline{D(A)}\,:\,\exists \delta_0>0,\ \forall \delta\in(0,\delta_0) \ \right.\nnn\\ 
&& ~~~~~~~~~~~~~~~~~~~~~~~~~ \left.\mbox{s.t.}\ u\in D((-A+\delta)^s),~\lim_{\delta\to 0^+}(-A+\delta)^s u\ \mbox{exists} \right\}
\eee
\end{itemize}
\end{definition}

It is worth noting that the above definition yields a well-defined closed operator which extends the original Balakrishnan operator. Moreover, in \cite{meichsner2017fractional}, it was shown that for a densely defined linear operator $A,$ we have
\bb{bala1}
\overline{J_{-A}^s} = (-A)^s.
\ee
There is still a wealth of information that one could present to provide background on fractional powers of operators, but these ideas will not be pertinent to the study at hand. Interested readers should see \cite{martinez2001theory} for more details.

Based on Assumption 2.1, we are in fact considering a special case of that studied in \cite{Gale2013}. Though the following result was included in \cite{Gale2013}, the proof was omitted for this special case, hence, we include a proof for completeness. The methods employed below are similar to those in \cite{meichsner2017fractional}.

\begin{theorem}
Let $u_0\in D((-A)^s).$ Then a solution to $(\ref{b1})$ is given by
\bb{usol1}
u(t) = \frac{1}{\Gamma(s)}\int_0^\infty z^{s-1}e^{-t^2/4z}T(z)(-A)^su_0\,dz
\ee
and also satisfies
\bb{uder}
\lim_{t\to 0^+} t^{1-2s}u'(t) = c_s (-A)^su_0,
\ee
where $c_s\mathrel{\mathop:}= 2^{1-2s}\Gamma(1-s)/\Gamma(s).$
\end{theorem}

\begin{proof}
We first show $(\ref{usol1})$ by demonstrating directly that $(\ref{usol1})$ satisfies $(\ref{b1})$ for $t>0.$ To that end we have
$$u'(t) = \frac{1}{\Gamma(s)}\int_0^\infty \left[\frac{-t}{2z}\right]z^{s-1}e^{-t^2/4z}T(z)(-A)^su_0\,dz$$
and
$$u''(t) = \frac{1}{\Gamma(s)}\int_0^\infty \left[\frac{-1}{2z} + \frac{t^2}{4z^2}\right]z^{s-1}e^{-t^2/4z}T(z)(-A)^su_0\,dz.$$
Thus, for $t>0$ we have
\bbb
u''(t) + \frac{\alpha}{t}u'(t) & = & \frac{1}{\Gamma(s)}\int_0^\infty \left[\frac{s-1}{z} + \frac{t^2}{4z^2}\right] z^{s-1}e^{-t^2/4z}T(z)(-A)^su_0\,dz\nnn\\
& = & \frac{1}{\Gamma(s)}\int_0^\infty \left[\frac{d}{dz}\left(z^{s-1}e^{-t^2/4z}\right)\right]T(z)(-A)^su_0\,dz.\label{lem22a}
\eee
By noting that $\textstyle\frac{d}{dz}T(z)(-A)^su_0 = AT(z)(-A)^su_0$ and employing integration by parts, we have
\bbbb
u''(t) + \frac{\alpha}{t}u'(t) & = & \frac{1}{\Gamma(s)}\int_0^\infty z^{s-1}e^{-t^2/4z}AT(z)(-A)^su_0\,dz\\
& = & -Au(t),
\eeee
which shows that $(\ref{usol1})$ is a solution to $(\ref{b1}).$ 

Next we show $(\ref{uder})$ holds for $u_0\in D((-A)^s).$ Note that
\bbb
t^{1-2s}u'(t) & = & \frac{1}{\Gamma(s)}\int_0^\infty \left(\frac{-t^{2-2s}}{2z}\right)z^{s-1}e^{-t^2/4z}T(z)(-A)^su_0\,dz\nnn\\
& = & \frac{c_st^{-2s}}{4^{-s}\Gamma(1-s)} \int_0^\infty z^{s}\left(\frac{d}{dz}e^{-t^2/4z}\right)T(z)(-A)^su_0\,dz.\label{lim1b}
\eee
By employing integration by parts in $(\ref{lim1b})$ we obtain
\bbb
&&\int_0^\infty z^{s}\left(\frac{d}{dz}e^{-t^2/4z}\right)T(z)(-A)^su_0\,dz\nnn\\ 
&& ~~~~~~~~~~~~~~~~~~~~ = -\int_0^\infty e^{-t^2/4z}\left[sz^{s-1}T(z) - z^sAT(z)\right](-A)^su_0\,dz, \label{lim1c}
\eee
where the boundary terms in $(\ref{lim1c})$ go to zero due to Assumption 2.1. After substituting $(\ref{lim1c})$ into $(\ref{lim1b})$ and employing integration by parts yet again, we obtain
\bbb
t^{1-2s}u'(t) & = & \frac{-c_st^{-2s}}{4^{-s}\Gamma(1-s)}\int_0^\infty e^{-t^2/4z}\left[sz^{s-1}T(z) - z^sAT(z)\right](-A)^su_0\,dz\nnn\\
& = & \frac{c_s}{\Gamma(-s)}\int_0^\infty z^{-s-1}e^{-t^2/4z}T(z)u_0\,dz\nnn\\
&& ~~~~~~~~~~~~~~~ - \frac{c_st^2}{4\Gamma(1-s)}\int_0^\infty z^{-s-2}e^{-t^2/4z}T(z)u_0\,dz\nnn\\
& = & \frac{c_s}{\Gamma(-s)}\int_0^\infty z^{-s-1}e^{-t^2/4z}\left[T(z)u_0-u_0\right]\,dz\nnn\\
&& ~~~~~~~~~~~~~~~ - \frac{c_st^2}{4\Gamma(1-s)}\int_0^\infty z^{-s-2}e^{-t^2/4z}\left[T(z)u_0 - u_0\right]\,dz.\label{lim2}
\eee
Note that the first integral in $(\ref{lim2})$ is Bochner integrable and is in fact the desired result, as
\bbb
\lim_{t\to 0^+}\frac{c_s}{\Gamma(-s)}\int_0^\infty z^{-s-1}e^{-t^2/4z}\left[T(z)u_0-u_0\right]\,dz & = & \frac{c_s}{\Gamma(-s)}\int_0^\infty z^{-s-1}\left[T(z)u_0-u_0\right]\,dz\nnn\\
& = & c_s(-A)^su_0\label{lim3}
\eee
by Definition 2.1 and $(\ref{bala1}).$ Thus, $(\ref{uder})$ follows if the second integral in $(\ref{lim2})$ goes to zero. By Taylor expanding T(z) about $z=0$ and employing Assumption 2.1, we have
\bbb
\left\|\frac{c_st^2}{4\Gamma(1-s)}\int_0^\infty z^{-s-2}e^{-t^2/4z}\left[T(z)u_0 - u_0\right]\,dz\right\| & \le & \frac{c_st^2\|Au_0\|}{4\Gamma(1-s)}\int_0^\infty z^{-s-1}e^{-t^2/4z}\,dz\nnn\\
& = & \frac{c_st^{2-2s}\|Au_0\|}{4^{-s}\Gamma(1-s)}\int_0^\infty y^{s-1}e^{-y}\,dy\nnn\\
& = & 2t^{2-2s}\|Au_0\|.\label{lim4}
\eee
Since $D(A)\subseteq D((-A)^s),$ it follows that $\|Au_0\|$ is bounded by assumption and we have
\bb{lim5}
\lim_{t\to 0^+}\left\|\frac{c_st^2}{4\Gamma(1-s)}\int_0^\infty z^{-s-2}e^{-t^2/4z}\left[T(z)u_0 - u_0\right]\,dz\right\| = 0.
\ee
Combining $(\ref{lim3})$ and $(\ref{lim5})$ yields the desired result. 
\end{proof}

\begin{lemma}
For all $u_0 \in X,$ a solution to $(\ref{b1})$ is given by $(\ref{usol1}).$
\end{lemma}

\begin{proof}
This result follows from \cite{meichsner2017fractional} and the fact that $A$ is densely defined in $X.$ 
\end{proof}

It is worth noting that the solution to $(\ref{b1})$ is not known to be unique in the current setting. However, there are numerous results regarding situations in which uniqueness is known, for instance, when $A$ is the Laplacian with Dirichlet boundary conditions, when $A$ is a sectorial operator and $X$ is a Hilbert space, and for all operators $A$ when $s=1/2$ \cite{doi:10.1080/03605302.2017.1363229,Gale2013,meichsner2017fractional,doi:10.1080/03605301003735680}. In order to avoid these issues in our ensuing analysis, we introduce the following assumption.

\begin{assume}
The solution to $(\ref{b1})$ is unique in $X$ for all $u_0\in D((-A)^s),$ and hence, given by $(\ref{usol1}).$
\end{assume}

\begin{rem}
One could replace Assumption 2.2 with the condition that the solution to $(\ref{b1})$ decays to zero at infinity, in some weak sense, and obtained the same result.
\end{rem}

\begin{rem}
It is worth noting that $(\ref{usol1})$ may seem a bit strange as it involves the fractional power of the operator $-A.$ However, as discussed in the Introduction, when considering $(\ref{e1})$ or $(\ref{b1})$ in applications the function $(-A)^su_0$ will be given. 
\end{rem}

We close this section by outlining some basic properties regarding the semigroup generated by $A$ and the operator $(-A)^s.$ For more details, see \cite{henry2006geometric}.

\begin{lemma}
Let $\alpha\ge 0$ and $0\le \gamma \le 1.$ Then there exists a constant $C>0$ such that
\begin{itemize}
\item[i.] $\|(-A)^\alpha T(t)\| \le Ct^{-\alpha},$ for $t>0;$
\item[ii.] $(-A)^\alpha T(t) = T(t)(-A)^\alpha,$ on $D((-A)^\alpha);$
\item[iii.] If $\alpha\ge \gamma,$ then $D((-A)^\alpha\subseteq D((-A)^\gamma).$
\end{itemize}
\end{lemma}

\section{Derivation of Approximation} \clearallnum

The properties of various forms of $(\ref{usol1})$ have been studied by several authors, see \cite{doi:10.1080/03605302.2017.1363229,Gale2013,meichsner2017fractional,doi:10.1080/03605301003735680} and the references therein. However, in this paper, we are concerned with the development and analysis of an approximation to $(\ref{usol1}).$ To that end, fix $0\ll M< N < \infty$ with $M,N\in\NN.$ Then we define $h\mathrel{\mathop:}=M/(N+1)$ and set $z_k\mathrel{\mathop:}= kh,\ k=0,\ldots,N+1.$ For the sake of the ensuing analysis, we define
\bb{a1}
u^{(M)}(t)\mathrel{\mathop:}= \frac{1}{\Gamma(s)}\int_0^M z^{s-1}e^{-t^2/4z}T(z)(-A)^su_0\,dz.
\ee
We then have the following result.

\begin{lemma}
Let $u$ and $u^{(M)}$ be defined as in $(\ref{usol1})$ and $(\ref{a1}),$ respectively. Then, if $u_0\in D((-A)^s),$ we have
$$\|u(t) - u^{(M)}(t)\| \le C
M^{s-1}e^{\mu(A)M}\|(-A)^su_0\|,\quad t\in [0,\infty),$$
where $C$ is a constant independent of $s$ and $M.$ 
\end{lemma}

\begin{proof}
We proceed by direct calculation. We have
\bbb
\left\|u(t) - u^{(M)}(t)\right\| & = & \frac{1}{\Gamma(s)}\left\|\int_0^\infty z^{s-1}e^{-t^2/4z}T(z)(-A)^su_0\,dz\right.\nnn\\
&& ~~~~~~~~~~~~~~~~~~~~ \left. - \int_0^M z^{s-1}e^{-t^2/4z}T(z)(-A)^su_0\,dz\right\|\nnn\\
& = & \frac{1}{\Gamma(s)}\left\|\int_M^\infty z^{s-1}e^{-t^2/4z}T(z)(-A)^su_0\,dz\right\|\nnn\\
& \le & \frac{e^{-t^2/4M}\|(-A)^su_0\|}{\Gamma(s)}\int_M^\infty z^{s-1}e^{\mu(A)z}\,dz\nnn\\
& \le & \frac{\|(-A)^su_0\|}{\Gamma(s)}\int_M^\infty z^{s-1}e^{\mu(A)z}\,dz,\label{lemma1a}
\eee
where $\mu(A)<0$ by Assumption 2.1 and is defined in $(\ref{lognorm}).$
It now remains to bound the integral in $(\ref{lemma1a}).$ To that end, let $w = -\mu(A)z,$ to obtain
\bb{1int1}
\int_M^\infty z^{s-1}e^{\mu(A)z}\,dz = \left(-\mu(A)\right)^{-s}\int_{-\mu(A)M}^\infty w^{s-1}e^{-w}\,dw.
\ee
Employing the substitution $y = w + \mu(A)M$ in $(\ref{1int1})$ then yields
\bb{1int2}
\left(-\mu(A)\right)^{-s}\int_{-\mu(A)M}^\infty w^{s-1}e^{-w}\,dw = \left(-\mu(A)\right)^{-1}M^{s-1}e^{\mu(A)M}\left[1 + R(M)\right],
\ee
where
$$R(M)\mathrel{\mathop:}= \frac{\Gamma(2-s)}{\mu(A)M\Gamma(s)}\int_0^\infty ye^{-y}\left(\int_0^1 \left(1 - \frac{y\xi}{\mu(A)M}\right)^{s-2}\,d\xi\right)\,dy.$$
Since $s\in (0,1),$ we have
\bb{1int3}
\|R(M)\| \le \frac{\Gamma(2-s)}{|\mu(A)|M\Gamma(s)}\int_0^\infty\int_0^1 ye^y\,d\xi\,dy = \frac{\Gamma(2-s)}{|\mu(A)|M\Gamma(s)}.
\ee
By substituting $(\ref{1int1})$-$(\ref{1int3})$ into $(\ref{lemma1a}),$ we obtain
$$\left\|u(t) - u^{(M)}(t)\right\| \le CM^{s-1}e^{\mu(A)M}\|(-A)^su_0\|.$$
\end{proof}

\begin{rem}
While the constant, $C,$ in Lemma 3.1 is independent of $s,$ $t,$ and $M,$ it does depend on $\mu(A).$ Hence, in practice, if $\mu(A)$ is close to zero we will need to increase the value of $M$ to obtain the desired accuracy.
\end{rem}

Now that it has been established that the truncation of the integral will not have drastic effects on the approximation of $(\ref{usol1}),$ for large $M,$ we present the proposed approximation to $u(t)$ in $X.$ For $u_0\in D((-A)^s),$ we define the following continuous in time approximation to $(\ref{usol1})$
\bb{approx}
v(t)\mathrel{\mathop:}= \sum_{k=0}^N \gamma_k(t) S^k(h)S(h/2)(-A)^su_0,
\ee
where
$$
\gamma_k(t)\mathrel{\mathop:}= \frac{h^s[(k+1)^s-k^s]}{\Gamma(1+s)}e^{-t^2/4z_{k+1/2}}\quad\mbox{and}\quad
S(h)\mathrel{\mathop:}= \left(I-\frac{h}{2}A\right)^{-1}\left(I+\frac{h}{2}A\right),
$$
with $z_{k+1/2}\mathrel{\mathop:}= z_k + h/2.$ 

\begin{assume}
Let the operator $S(w)$ be nonexpansive in $X.$
\end{assume}

\begin{rem}
Note that if $X$ is in fact a Hilbert space, then the assumption that $S(w)$ is nonexpansive will always be valid due to Assumption 2.1. That is, let $v\in D(A),$ then
\bbbb
\left\|\left(I+\frac{w}{2}A\right)v\right\|^2 & = & \|v\|^2 + w\mbox{Re}[ Av,v] + \frac{w^2}{4}\|Av\|^2\\
& \le & \|v\|^2 - w\mbox{Re}[ Av,v] + \frac{w^2}{4}\|Av\|^2\\
& = & \left\|\left(I - \frac{w}{2}A\right)v\right\|^2,
\eeee
where $[ \cdot,\cdot]$ is the inner product on the Hilbert space $X.$ However, in general Assumption 3.1 does not hold in general Banach spaces, $X,$ with $A$ satisfying Assumption 2.1. For more details see \cite{barbu2010nonlinear}.
\end{rem}

We begin by demonstrating that $(\ref{approx})$ is stable in $X.$

\begin{lemma}
Let $u_0\in D((-A)^s)$ and $v(t)$ be as defined in $(\ref{approx}).$ Then we have
$$\|v(t)\| \le C\|(-A)^su_0\|,\quad t\in [0,\infty),$$
where $C$ is a constant independent of $h$ and $k=0,\ldots,N.$
\end{lemma}

\begin{proof}
Direct calculation yields
\bb{st1}
\|v(t)\| \le \sum_{k=0}^N \|\gamma_k(t)\|\left\|S(h)\right\|^k\left\|S(h/2)\right\|\|(-A)^su_0\| \le \sum_{k=0}^N \|\gamma_k(t)\|\|(-A)^su_0\|,
\ee
by Assumption 3.1. Thus it we need only consider the $\gamma_k(t).$ We first note that since $\gamma_k(t)\ge 0$ for $t\ge 0,$ we have
\bbb
\sum_{k=0}^N \|\gamma_k(t)\| & = & \frac{h^s}{\Gamma(1+s)}\sum_{k=0}^{N}[(k+1)^s-k^s]e^{-t^2/4z_{k+1/2}}\nnn\\ 
& = & \frac{h}{\Gamma(s)}\sum_{k=0}^N \xi_k^{s-1}e^{-t^2/4z_{k+1/2}},\label{new1}
\eee
for some $\xi_k\in(z_k,z_{k+1}),\ k=0,...,N.$ Now let $\beta\in\NN$ be the index such that $z_\beta \le 1 < z_{\beta+1}.$ We then have, by $(\ref{new1}),$
\bbbb
\sum_{k=0}^N \|\gamma_k(t)\| & = & \frac{h}{\Gamma(s)}\sum_{k=0}^{N}\xi_k^{s-1}e^{-t^2/4z_{k+1/2}}\nnn\\ 
& = & \frac{h}{\Gamma(s)}\sum_{k=0}^{\beta-1} \xi_k^{s-1}e^{-t^2/4z_{k+1/2}} + \frac{h}{\Gamma(s)}\sum_{k=\beta}^{N} \xi_k^{s-1}e^{-t^2/4z_{k+1/2}}\nnn\\
& \le & \frac{h}{\Gamma(s)}\sum_{k=0}^{\beta-1} \xi_k^{s-1}e^{-t^2/4z_{k+1/2}} + \frac{h}{\Gamma(s)}\sum_{k=\beta}^{N} e^{-t^2/4z_{k+1/2}}\label{st1d}
\eeee
By noting that
$$\frac{h}{\Gamma(s)}\sum_{k=0}^{\beta-1} z_{k+1/2}^{s-1}e^{-t^2/4z_{k+1/2}} ~~\le~~ \frac{h}{\Gamma(s)}\sum_{k=0}^{\beta-1} e^{-t^2/4} ~~\le~~ e^{-t^2/4},$$
we then obtain
\bb{st1e}
\frac{h}{\Gamma(s)}\sum_{k=0}^{N}\xi_k^{s-1}e^{-t^2/4z_{k+1/2}} ~~\le~~ e^{-t^2/4} + \frac{h}{\Gamma(s)}\sum_{k=\beta + 1}^{N} e^{-t^2/4z_{k+1/2}} ~~\le~~ C,
\ee
where $C$ is a constant independent of $h$ and $k=0,\ldots,N.$ Substituting $(\ref{st1e})$ into $(\ref{st1})$ completes the proof.
\end{proof}

\section{Convergence} \clearallnum

We now present several convergence estimates that will be necessary for the main result of this section. 

\begin{lemma}
Let $u_0\in D((-A)^{\ell+s}),\ \ell=1,2.$ Then we have
$$\left\|\frac{1}{\Gamma(s)}\int_{z_k}^{z_{k+1}} z^{s-1}e^{-t^2/4z}T(z)(-A)^su_0\,dz - \gamma_k(t)T(z_{k+1/2})(-A)^su_0\right\| \le Ch^{\ell+s},$$
for $t\in [0,\infty),$
where $C$ is a constant independent of $h$ and $k=0,...,N.$
\end{lemma}

\begin{proof}
We begin by defining
\bb{ffff}
f(z)\mathrel{\mathop:}= e^{-t^2/4z}T(z)(-A)^su_0,
\ee
for all $t\in [0,\infty).$ Thus, by employing $(\ref{ffff}),$ we have by direct calculation
\bbb
&&\int_{z_k}^{z_{k+1}} z^{s-1}f(z)\,dz\nnn\\ 
&& ~~~~~~~~~~ = \int_{z_k}^{z_{k+1}} z^{s-1}\left[f(z_{k+1/2}) + f'(z_{k+1/2})(z-z_{k+1/2}) + \frac{f''(\xi_k)}{2}(z-z_{k+1/2})^2\right]\,dz \nnn\\
&& ~~~~~~~~~~ = \frac{z_{k+1}^s - z_k^s}{s}f(z_{k+1/2}) + \frac{f''(\xi_k)}{2}\int_{z_k}^{z_{k+1}} z^{s-1}(z-z_{k+1/2})^2\,dz,\label{ffff1}
\eee
for some $\xi_k\in(z_k,z_{k+1}).$ By noting that
$$\frac{z_{k+1}^s - z_k^s}{s}f(z_{k+1/2}) = \Gamma(s)\gamma_k(t)T(z_{k+1/2})(-A)^su_0,$$
it is clear from $(\ref{ffff1})$ that it only remains to show
$$\left\|\frac{f''(\xi_k)}{2}\int_{z_k}^{z_{k+1}} z^{s-1}(z-z_{k+1/2})^2\,dz\right\| \le Ch^{\ell+s+1},$$
for $\ell = 1,2.$ We first note that, by Hille's theorem, we have
\bb{ffff2}
f''(z) = \left(\frac{t^2}{4z^2}I - A\right)^2e^{-t^2/4z}T(z)(-A)^su_0 - \frac{t^2}{4z^3}e^{-t^2/4z}T(z)(-A)^su_0.
\ee
From $(\ref{ffff2}),$ Assumption 2.1, and Lemma 2.3, we have
\bb{fffff}
\|f''(z)\| \le t^{\ell-2},
\ee
since $z^{-\eta}e^{-t^2/4z}$ is bounded for $z\in (0,\infty)$ and $\eta = 2,3,4.$ 

We next note that
\bb{ffff3}
\left\|\int_{z_k}^{z_{k+1}} z^{s-1}(z-z_{k+1/2})^2\,dz\right\| ~~\le~~ \frac{h^2}{4}\int_{z_k}^{z_{k+1}} z^{s-1}\, dz ~~=~~ \frac{h^{2+s}[(k+1)^s-k^s]}{4s}.
\ee
Combining $(\ref{fffff})$ and $(\ref{ffff3})$ into $(\ref{ffff1})$ then yields
\bbbb
&& \left\|\frac{1}{\Gamma(s)}\int_{z_k}^{z_{k+1}} z^{s-1}e^{-t^2/4z}T(z)(-A)^su_0\,dz - \gamma_k(t)T(z_{k+1/2})(-A)^su_0\right\|\\
&& ~~~~~~~~~~~~~~~~~~~~~~~~~~~~~~ = \left\|\frac{f''(\xi_k)}{\Gamma(s)2}\int_{z_k}^{z_{k+1}} z^{s-1}(z-z_{k+1/2})^2\,dz\right\|\\
&& ~~~~~~~~~~~~~~~~~~~~~~~~~~~~~~ \le \frac{\xi_k^{\ell-2}h^{2+s}[(k+1)^s-k^s]}{8\Gamma(1+s)}\\
&& ~~~~~~~~~~~~~~~~~~~~~~~~~~~~~~ \le Ch^{\ell+s},
\eeee
where $\ell=1,2$ for $u_0\in D((-A)^{\ell+s}),$ respectively, and $C$ is a constant independent of $h$ and $k=0,...,N.$ 
\end{proof}

We now define a family of operators which will be employed in the proof of the following lemma.
Let $\Lambda_k\,:\,[0,\infty)\to B(X)$ be defined as
\bb{lam1}
\Lambda_k(z) \mathrel{\mathop:}= z^{-k}\int_0^z T(z-y)\frac{y^{k-1}}{(k-1)!}\,dy,\quad k\ge 1,
\ee
with $\Lambda_0(z)\mathrel{\mathop:}= T(z).$ The $\Lambda_k$ satisfy the following recurrence relation
\bb{lam2}
\Lambda_k(z) = \frac{1}{k!}I + zA\Lambda_{k+1}(z),\quad k\ge 0.
\ee

\begin{lemma}
Let $u_0\in D((-A)^{\ell+s}),\ \ell = 1,2.$ Then we have
$$\left\|T(h)(-A)^su_0 - S(h)(-A)^su_0\right\| \le Ch^{\ell+1},$$
where $C$ is a constant independent of $h$ and $k=0,\ldots,N.$ 
\end{lemma}

\begin{proof}
By employing $(\ref{lam1})$ and $(\ref{lam2}),$ we readily obtain
\bbb
&&\left(I-\frac{h}{2}A\right)\left[T(h)(-A)^su_0 - S(h)(-A)^su_0\right] \nnn\\
&& ~~~~~~~~~~~~~~~~~~~~ =  \left[\left(I-\frac{h}{2}A\right)\left(I + hA\Lambda_1(h)\right) - \left(I+\frac{h}{2}A\right)\right](-A)^su_0\nnn\\
&& ~~~~~~~~~~~~~~~~~~~~ = hA\left[\Lambda_1(h) - \frac{h}{2}A\Lambda_1(h) - I\right](-A)^su_0\nnn\\
&& ~~~~~~~~~~~~~~~~~~~~ = \frac{h^2}{2}A^2\left[2\Lambda_2(h) - \Lambda_1(h)\right](-A)^su_0\label{semi1}
\eee
for $u_0\in D((-A)^{2+s}).$ Further, consideration of $(\ref{semi1})$ gives
\bbb
2\Lambda_2(h) - \Lambda_1(h) & = & \int_0^h \left[\frac{2y-h}{h^2}\right]T(h-y)\,dy\nnn\\
& = & \int_0^1 (2w-1)T(h - hw)\,dw\nnn\\
& = & hA\int_0^1 w(w-1)T(h-hw)\,dw.\label{semi1a}
\eee
Combining $(\ref{semi1})$ and $(\ref{semi1a})$ finally gives
\bbb
&& T(h)(-A)^su_0 - S(h)(-A)^su_0\nnn\\
&& ~~~~~~~~~~~~~~~ = \frac{h^3}{2}A^3\left(I-\frac{h}{2}A\right)^{-1}\left[\int_0^1 w(w-1)T(h-hw)\,dw\right](-A)^su_0. \label{semi1c}
\eee
The result for $\ell=1$ then follows by taking the norm of both sides of $(\ref{semi1c}).$
Since $A(I-A)^{-1}$ is a bounded operator on $X,$ it follows that $(\ref{semi1c})$ holds for $u_0\in D(A^{1+s}),$ and taking the norm of both sides yields the result for $\ell=2.$ 
\end{proof}

Combining the above results, we have the following theorem demonstrating the desired convergence rate of the proposed scheme. 

\begin{theorem}
Let $u_0\in D((-A)^{\ell + s}),\ \ell = 1,2.$ Then, we have
$$\left\|u(t) - v(t)\right\| \le C\left(M^{s-1}e^{\mu(A)M} + h^{\ell+s-1}\right),\quad t\in [0,\infty),$$
where $C$ is a constant independent of $h$ and $k=0,\ldots,N.$
\end{theorem}

\begin{proof}
We begin by writing the difference as follows
\bbbb
&& u(t) - v(t)\\ 
&& ~~~~~ = \underbrace{\frac{1}{\Gamma(s)}\int_0^\infty z^{s-1}e^{-t^2/4z}T(z)(-A)^su_0\,dz - \frac{1}{\Gamma(s)}\int_0^M z^{s-1}e^{-t^2/4z}T(z)(-A)^su_0\,dz}_{\mathrel{\mathop:}= \,I_1}\\
&& ~~~~~~~~~~ + \underbrace{\frac{1}{\Gamma(s)}\int_0^M z^{s-1}e^{-t^2/4z}T(z)(-A)^su_0\,dz - \sum_{k=0}^N \gamma_k(t)T(z_{k+1/2})(-A)^su_0}_{\mathrel{\mathop:}= \,I_2}\\
&& ~~~~~~~~~~ + \underbrace{\sum_{k=0}^N \gamma_k(t)T(z_{k+1/2})(-A)^su_0 - \sum_{k=0}^N \gamma_k(t) S^k(h)S(h/2)(-A)^su_0}_{\mathrel{\mathop:}= \,I_3}\\
&& ~~~~~ = I_1 + I_2 + I_3.
\eeee
We will now apply the previous lemmas to $I_1,\ I_2,$ and $I_3.$ First, by Lemma 3.1, it is immediately clear that we have
\bb{est1}
\|I_1\| \le CM^{s-1}e^{\mu(A)M}\|(-A)^su_0\|,
\ee
and thus $(\ref{est1})$ may be made arbitrarily small for large enough $M.$ Next, by Lemma 4.1 we have
\bbb
\|I_2\| & = & \left\|\sum_{k=0}^N\left[\frac{1}{\Gamma(s)}\int_{z_k}^{z_{k+1}} z^{s-1}e^{-t^2/4z}T(z)(-A)^su_0\,dz - \gamma_k(t)T(z_{k+1/2})(-A)^su_0\right]\right\|\nnn\\
& \le & \sum_{k=0}^N\left\|\frac{1}{\Gamma(s)}\int_{z_k}^{z_{k+1}} z^{s-1}e^{-t^2/4z}T(z)(-A)^su_0\,dz - \gamma_k(t)T(z_{k+1/2})(-A)^su_0\right\|\nnn\\
& \le & \sum_{k=0}^N Ch^{\ell+s+1}\nnn\\
& \le & Ch^{\ell+s} \label{est2},
\eee
where the value of $C$ changes throughout the calculations, but remains independent of $h$ and $k=0,...,N,$ with $\ell = 1$ or $\ell =2,$ depending on the regularity of the initial data. Finally, by Lemma 4.2, Assumption 2.1, and Assumption 3.1,  we have
\bbb
\|I_3\| & = & \left\|\sum_{k=0}^N\gamma_k(t) \left[T^k(h)T(h/2) - S^k(h)S(h/2)\right](-A)^su_0\right\|\nnn\\
& = & \left\|\sum_{k=0}^N\gamma_k(t) \left[T^k(h)T(h/2) - T^k(h)S(h/2)\right.\right.\nnn\\
&& ~~~~~~~~~~~~~~~ \left.\left.\vphantom{\sum_{k=0}^N} + T^k(h)S(h/2) - S^k(h)S(h/2)\right](-A)^su_0\right\|\nnn\\
& \le & \sum_{k=0}^N\gamma_k(t) \|T^k(h)\|\|T(h/2) - S(h/2)\|\|(-A)^su_0\|\nnn\\
&& ~~~~~  + \sum_{k=0}^N \gamma_k(t) \sum_{j=1}^k \|T^{k-j}(h)\|\left\|T(h) - S(h)\right\|\|S^{j-1}(h)\|\|S(h/2)\|\|(-A)^su_0\|\nnn\\
& \le & \sum_{k=0}^N \gamma_k(t) \left[Ch^{\ell + 1} + \sum_{j=1}^k Ch^{\ell+1} \right]\nnn\\
& \le & C\sum_{k=0}^N\gamma_k(t)(1+k)h^{\ell+1}\nnn\\
& \le & Ch^\ell, \quad\label{est3}
\eee
where, once again, $C$ is independent of $h$ and $k=0,...,N,$ with $\ell$ as before. Combining the obtained bounds in $(\ref{est1})$-$(\ref{est3})$ yields the desired result. 
\end{proof}

\begin{rem}
The first term in our estimate
$$M^{s-1}e^{\mu(A)M}$$
is an expected and standard term (a similar term appears in, for example, \cite{bonito2019numerical,chen2015pde}). However, we intend to remove this estimate in future work by relating the choice of $M$ to the parameter $h.$
\end{rem}

\section{Numerical Experiments} \clearallnum

In order to further verify the sharpness of our analysis and demonstrate the efficacy of the proposed method, we present some numerical experiments. The first and second examples consider a one-dimensional and two-dimensional problem, respectively, where the given operator is the Laplace operator. In both cases, we see that the expected theoretical convergence results are recovered. The third example outlines how our theory easily includes general strongly elliptic differential operators which still satisfy the theoretical results from Theorem 4.1.

It is the case that the nonlocality of fractional problems can increase memory requirements to the point that some standard methods become quite inefficient. To that end, we briefly explore the memory requirements needed to implement the proposed method $(\ref{approx}).$ If we fix a particular numerical approximation of the operator $A,$ then our method can be shown to be $\mathcal{O}(N).$ Initially, it may seem that the computations are more involved, due to the evaluation of the operator $S^k(h)$ (seeming to require $N$ additions and $(N^2+N)/2$ multiplications), but these evaluations can be be completed with only $N$ additions and $N$ multiplications by employing nested operations similar to the classical Horner algorithm for polynomials. To determine an exact efficiency, one must consider the precise discretization method employed to approximate the operator $A.$ Due to the inversion needed to compute $S,$ it is clear that discretizations resulting in a banded matrix structure will prove the most efficient. We leave more detailed studies of the efficiency, as well as computational run time trials, for future work.

\subsection{A One-Dimensional Example}

Let $\Omega\mathrel{\mathop:}= (-1,1)\subset\RR$ and define $X = L^2(\Omega)$ with its usual inner product and norm. Consider the problem
\bb{ex1}
\left\{\begin{array}{ll}
(-\Delta)^su = f, & \quad x\in \Omega,\\
u = 0, & \quad x \in \partial\Omega,
\end{array}\right.
\ee
where
$$f(x) = \left(\frac{\pi}{2}\right)^{2s}\sin\left(\frac{\pi(1+x)}{2}\right),$$
and the fractional Laplacian is defined via the spectral definition (which can be done due to the smoothness of $f$). That is,
\bb{spec_frac}
(-\Delta)^su(x) \mathrel{\mathop:}= \sum_{k\in\NN} \lambda_k^s u_k \varphi_k,
\ee
where the set $\{\varphi_k\}_{k\in\NN}\subset H_0^1(\Omega)$ forms an orthonormal basis of $X$ and satisfy
$$\left\{\begin{array}{ll}
(-\Delta)^s\varphi_k = \lambda_k\varphi_k, & \quad x\in \Omega,\\
\varphi_k = 0, & \quad x\in\partial\Omega,
\end{array}\right.$$
for each $k\in\NN,$ and
$$u_k \mathrel{\mathop:}= [ u,\varphi_k],\quad k\in\NN,$$
where $[\cdot,\cdot]$ is the usual inner product on $X.$ The true solution to $(\ref{ex1})$ can be found explicitly and is given by
\bb{ex1sol}
u(x) = \sin\left(\frac{\pi(1+x)}{2}\right).
\ee
This true solution may be obtained by using $(\ref{spec_frac}).$ We will first demonstrate how our method can approximate problems such as $(\ref{ex1})$ by taking the trace of $(\ref{approx}),$ and then consider the convergence properties of the approximation of the solution to the extension problem given by $(\ref{b1}).$ 

As has been demonstrated in this work, it is the case that $(\ref{ex1})$ can be recovered from the following extension problem:
\bb{ex2}
\left\{\begin{array}{ll}
v''(t)+\frac{\alpha}{t}v'(t) = -Av(t), & \quad t\in (0,\infty),\\
v(0) = u,
\end{array}\right.
\ee
where $\alpha = 1 - 2s$ and $A\mathrel{\mathop:}= \Delta\,:\,H_0^1(\Omega)\cap H^2(\Omega)\subset X \to L^2(\Omega).$ In order to implement the problem, we replace the operator $A$ with some finite dimensional discretization. For simplicity, we employ the standard second-order finite difference method on uniform grids \cite{Iserles:2008:FCN:1506265}. It is worth noting that the use of any other consistent spatial disretization method is acceptable, and as seen below, the convergence rate in the spatial direction will not be affected. The freedom to choose the discretization method for the operator $A$ is one of the primary benefits of this approach.

The tables in Fig. 5.1 demonstrate the numerical convergence rates of the scheme given by $(\ref{approx}).$ For these experiments, we chose a modest value of $M=100,$ in order to demonstrate the efficacy of the method with respect to the truncation. It is clear from this table that the numerical solution converges at the expected rates given by Theorem 4.1. For the sake of clarity, we present a plot of the numerical solution to $(\ref{ex2})$ in Fig. 5.2.

\begin{figure}[!htb]
\centering
$$\begin{array}{ccc}
\mbox{Spatial Convergence Rates} & & \mbox{Extended Variable Convergence Rates}\\
& & \\
\begin{array}{c|c}
\delta x & \mbox{Convergence Rate}\\
\hline
\hline
0.04000 & ---\\
0.02000 & 1.99676\\
0.01000 & 1.99913\\
0.00500 & 1.99976\\
0.00250 & 1.99993\\
0.00125 & 2.00007
\end{array}
& &
\begin{array}{c|c}
h & \mbox{Convergence Rate}\\
\hline
\hline
2.00\times 10^{-4} & --- \\
1.00\times 10^{-4} & 1.89877\\
5.00\times 10^{-5} & 1.90122\\
2.50\times 10^{-5} & 1.90430\\
1.25\times 10^{-5} & 1.91225\\
6.25\times 10^{-6} & 1.91275
\end{array}
\end{array}$$

\caption{Convergence rates of the proposed approximation method when $s=0.9$ and $M=100.$ The convergence in the original spatial variable is the expected second order and the extended variable convergence rate is approximately $1+s = 1.9,$ as predicted.}
\end{figure}

Finally, we finish this example by presenting convergence rates for several other choices of the parameter $s.$ These results are presented in Fig. 5.3 and further verify our theoretical results. For brevity, we omit the convergence rates for the parameter $h,$ as they are not affected by these choices and match those of Fig. 5.1.

\begin{figure}[htb!]
\centering

\includegraphics[scale=0.5]{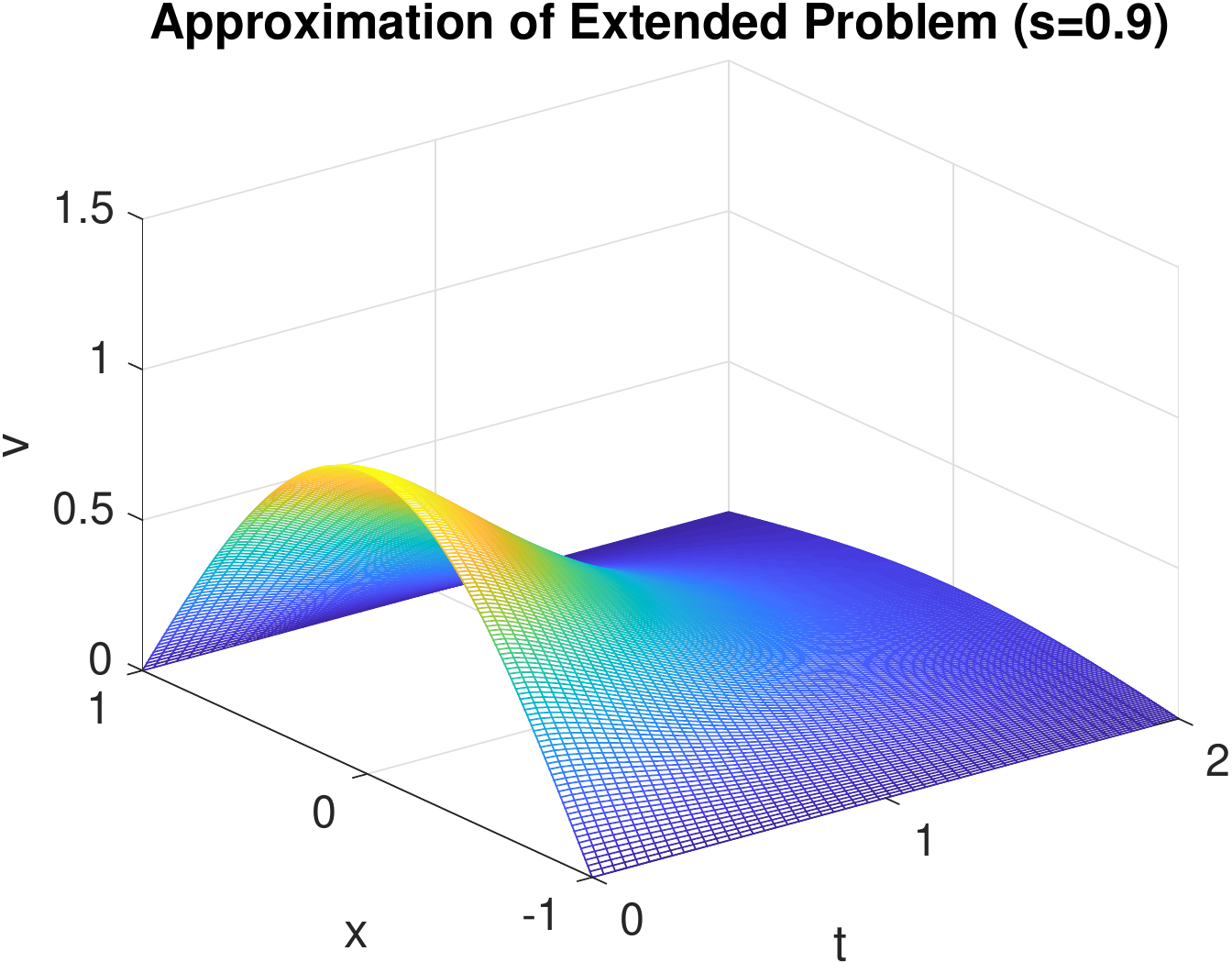}

\caption{A three-dimension surface representing an approximation to the solution, $v(x,t),$ of the extension problem $(\ref{ex2})$ on $[0,T].$ The rapid solution decay predicted by Lemma 3.1 is demonstrated clearly.}
\end{figure}

\begin{figure}[htb!]
\centering
$$
\begin{array}{c|c|c|c|c}
h & ~s = 0.8~ & ~s = 0.7~ & ~s = 0.6~ & ~s = 0.5~\\
\hline
\hline
~2.00\times 10^{-4}~ & --- & --- & --- & ---\\
~1.00\times 10^{-4}~ & ~1.79112~ & ~1.68825~ & ~1.58002~ & ~1.47790 \\
~5.00\times 10^{-5}~ & ~1.79982~ & ~1.68944~ & ~1.58344 & ~1.47999 \\
~2.50\times 10^{-5}~ & ~1.80012~ & ~1.69020~ & ~1.58566 & ~1.48212 \\
~1.25\times 10^{-5}~ & ~1.80323~ & ~1.70043~ & ~1.59111 & ~1.48892 \\
~6.25\times 10^{-6}~ & ~1.80895~ & ~1.70280~ & ~1.59988 & ~1.49489
\end{array}
$$

\caption{Extended variable convergence rates of the proposed method for various choices of the parameter $s,$ with $M=100.$ It is clear from the table that the expected convergence rate of order $1+s$ is approximately recovered, as expected.}
\end{figure}

\begin{rem}
We wish to note that our experiments have shown that more modest choices of $M$ may still provide accurate numerical solutions, as well. However, we leave a more detailed study of these choices for future endeavors.
\end{rem}

\begin{rem}
We also wish to note that one could potentially increase efficiency of the proposed numerical method by allowing it to be run in parallel. For instance, one may rewrite the method as
\bb{para1}
v(t) = \sum_{k=0}^{\lfloor N/2\rfloor} \gamma_{2k} S^{2k}(h)g + S(h)\sum_{k=0}^{\lfloor N/2\rfloor} \gamma_{2k+1}S^{2k}(h)g,
\ee
where 
$$g\mathrel{\mathop:}= S(h/2)(-A)^su_0.$$
While the implementation of $(\ref{para1})$ requires more direct evaluations than the original algorithm, one can easily compute $(\ref{para1})$ in parallel by noting that 
\bb{para2}
S^{2k}(h) = S^k(2h) + \mathcal{O}(h^2). 
\ee
While there is a slight decrease in accuracy due to employing $(\ref{para2}),$ it is the case that for complicated operators $A$ this could potentially be a favorable trade-off.
\end{rem}

\subsection{A Two-Dimensional Example}

Let $\Omega = \{(x,y)\in\RR^2\ : \ x^2+y^2<1\}.$ We then once again consider the problem $(\ref{ex1}),$ but with
$$f(x,y) = (\lambda_{1,1})^s \varphi_{1,1},$$
where $(\lambda_{1,1},\varphi_{1,1})$ are the first eigenpair of the operator $\Delta$ on $\Omega,$ and are given by
\bb{epair2}
\varphi_{1,1}(r,\theta) = J_1(j_{1,1}r)(A_{1,1}cos\theta + B_{1,1}\sin\theta)
\ee
and
\bb{epair1}
\lambda_{1,1} = j_{1,1}^2,
\ee
where $J_1$ is the first Bessel function of the first kind, $j_{1,1}$ is the first zero of $J_1,$ and $A_{1,1}$ and $B_{1,1}$ are real-valued constants that ensure $\|\varphi_{1,1}\|_X = 1.$

By using $(\ref{spec_frac})$ one may determine the true solution of the extension problem to be given by
\bb{true_circ}
u(r,\theta,t) = \frac{2^{1-s}t^2}{\Gamma(s)}(\lambda_{1,1})^{s/2}\varphi_{1,1}(r,\theta)K_s(\sqrt{2}\pi t),
\ee
where $K_s$ is the modified Bessel function of the second kind (see \cite{abramowitz1965handbook} for more information regarding the aforementioned Bessel functions and their properties). Moreover, it can be shown that $j_{1,1} \approx 3.8317.$ 

Just as before, we discretize the operator $A$ using standard finite differences (see, for example, \cite{JONES2016207} for more details on doing so for $\Omega$ as above). However, in this example we will only consider convergence aspects of the method in the extended variable direction. 
A table with convergence rates for various choices of the parameter $s$ are given in Fig. 5.4. As before, the results are in agreement with the rates given in Theorem 4.1.

\begin{figure}[htb!]
\centering
$$
\begin{array}{c|c|c|c|c}
h & ~s = 0.8~ & ~s = 0.6~ & ~s = 0.4~ & ~s = 0.2~\\
\hline
\hline
~2.00\times 10^{-4}~ & --- & --- & --- & ---\\
~1.00\times 10^{-4}~ & ~1.79112~ & ~1.58825~ & ~1.38002~ & ~1.17790 \\
~5.00\times 10^{-5}~ & ~1.79982~ & ~1.58944~ & ~1.38344 & ~1.17999 \\
~2.50\times 10^{-5}~ & ~1.80012~ & ~1.59020~ & ~1.38566 & ~1.18212 \\
~1.25\times 10^{-5}~ & ~1.80323~ & ~1.60043~ & ~1.39111 & ~1.18892 \\
~6.25\times 10^{-6}~ & ~1.80895~ & ~1.60280~ & ~1.39988 & ~1.19489
\end{array}
$$

\caption{Convergence rates for the numerical solution $v$ with respect to the extended value parameter $h.$ For these computations, we fixed $M=100.$ It is clear that the approximate theoretical convergence rate of $1+s$ is recovered.}
\end{figure}

\subsection{A General Second-Order Elliptic Operator}

As should be clear from the outlined theory, the proposed numerical algorithm will still be applicable when the operator $A$ is a general strongly elliptic operator on some domain $\Omega\subset \RR^d.$ Following the construction developed in \cite{chen2015pde}, let $\mathscr{L}$ be given by
\bb{L}
\mathscr{L}u \mathrel{\mathop:}= -\mbox{div}_x(A\nabla_x u) + cu,
\ee
where $x\in\Omega$ (here we emphasize the difference between the variable $x$ and the extended variable, $t.$). We assume that the function $c\in L^\infty(\Omega)$ with $c\ge 0$ almost everywhere, $A\in C^{0,1}(\Omega,\mbox{GL}(d,\RR))$ is symmetric and positive definite, and that $\Omega$ is sufficiently smooth. Then, it follows that given $f\in X\mathrel{\mathop:}= L^2(\Omega),$ there exists a unique $u\in D(\mathscr{L})\mathrel{\mathop:}= H_0^1(\Omega)\cap H^2(\Omega)$ such that
$$\left\{\begin{array}{ll}
\mathscr{L}u = f, & \quad x\in\Omega,\\
u = 0, & \quad x\in\partial\Omega.
\end{array}\right.$$

Moreover, we have that $\mathscr{L}$ satisfies Assumption 2.1. It is clear that the operator's domain is densely defined in $X.$ If we employ the same framework as that in Example 2.1, then we have
$$[-\mathscr{L}u,u] = -[A\nabla_x u, \nabla_x u] - [cu,u] \le -C[u,u],\quad \mbox{for all } u\in D(\mathscr{L}),$$
where $[\cdot,\cdot]$ is the standard inner product on $X,$ $C$ is a positive constant, and the final inequality follows from the fact that $A$ is positive definite and $c$ is positive almost everywhere. Thus, we have $\mbox{Re}[-\mathscr{L}u,u] < 0$ for all $u\in D(\mathscr{L}),$ as desired. Further, the range condition of a strictly $m-$dissipative operator is satisfied, as well. Thus, it follows that our approximation, given by $(\ref{approx}),$ of the Bessel-type problem
\bb{elip_b}
\left\{\begin{array}{ll}
v''(t) + \frac{\alpha}{t}v'(t) = -\mathscr{L}v(t), & \quad t\in(0,\infty),\\
v(0) = u \in X,
\end{array}\right.
\ee
is well-defined, stable, and convergent. Moreover, it follows that the initial datum, $u,$ will be the solution of the following nonlocal elliptic problem
\bb{elip_frac}
\left\{\begin{array}{ll}
\mathscr{L}^s u = f, & \quad x \in \Omega,\\
u = 0, & \quad x\in\partial\Omega.
\end{array}\right.
\ee
It should be noted that we have omitted some more technical details, such as the fact that the operator $\mathscr{L}^s$ may be extended to the fractional space $\mathbb{H}^s(\Omega),$ since such things are not relevant to the current discussion (this is because we are considering functions with higher regularity to guarantee the theoretical rates of convergence). However, interested readers may see \cite{Gale2013}, and the reference therein, for more details. 

In order to provide an explicit example, we let $d=1,$ $\Omega \mathrel{\mathop:}= (-1,1)\subset \RR,$ $A \equiv 1,$ $c\equiv 1,$ and $f(x) = \sin(\pi x).$ Since we do not have an explicit solution to $(\ref{elip_frac})$ on hand, we will approximate the numerical rate of convergence. Let $v_h$ denote the numerical solution obtained by $(\ref{approx})$ on a grid defined by the parameter $h.$ Then we define $v_{h/j}$ to be the numerical solution obtained by $(\ref{approx})$ on a grid with spacing $h/j,\ j=2,4.$ We then have the following Milne device for determining the numerical rate of convergence, denoted $p_h,$ given by
\bb{num_rate}
p_h \approx \frac{1}{\ln 2}\ln \frac{\|v_h - v_{h/2}\|}{\|v_{h/2} - v_{h/4}\|},
\ee
where the differences are defined on the grid associated to $h$ (the least refined grid) and the norm $\|\cdot\|$ is an appropriately chosen norm (in our case, the discrete $X-$norm).

\begin{rem}
The convergence of the numerical method given by $(\ref{approx})$ is also influenced by the error of truncating domain of the Bessel-type problem. However, by assuming this error is roughly constant, which is reasonable for a fixed value of $M$ and $s,$ the approximation given by $(\ref{num_rate})$ will be accurate.
\end{rem}

The table in Fig. 5.5 demonstrates the obtained numerical convergence rates for various choices of $s,$ with $M=100.$ We see from the table that the computed values still reflect the expected theoretical values obtained in Theorem 4.1. Moreover, we present the convergence rate for the case when $s=1,$ to verify the expected second-order convergence of our method.

\begin{figure}[htb!]
\centering
$$
\begin{array}{c|c|c|c|c|c}
h & ~s=1.0~ & ~s = 0.9~ & ~s = 0.8~ & ~s = 0.7~ & ~s = 0.6~ \\
\hline
\hline
~2.00\times 10^{-4}~ & --- & --- & --- & --- & --- \\
~1.00\times 10^{-4}~ & ~1.99825~ & ~1.88933~ & ~1.78458~ & ~1.67989 & 1.58991\\
~5.00\times 10^{-5}~ & ~1.99914~ & ~1.89345~ & ~1.79111 & ~1.68631 & 1.59944\\
~2.50\times 10^{-5}~ & ~2.00031~ & ~1.89844~ & ~1.80001 & ~1.68946 & 1.60003\\
~1.25\times 10^{-5}~ & ~2.00121~ & ~1.90034~ & ~1.80143 & ~1.69526 & 1.60139\\
~6.25\times 10^{-6}~ & ~2.00317~ & ~1.90127~ & ~1.80177 & ~1.70020 & 1.60155
\end{array}
$$

\caption{Convergence rates of the numerical method when applied to a general elliptic problem, per the derived Milne device. We fix $M=100$ and consider various values of $s.$ The results agree with the theoretical values obtained in Theorem 4.1.}
\end{figure}


\section{Concluding Remarks}

In this article, we consider an approximation to an abstract extension problem, that is the generalization of the method demonstrated by Caffarelli and Silvestre in \cite{doi:10.1080/03605300600987306}. This generalized setting allows for the consideration of a wider class of problems, while also allowing for the use of robust nonlinear functional analytic techniques. By taking this novel approach, we develop methods which differ from those presented in \cite{chen2015pde,d2013fractional,huang2014numerical,nochetto2015pde}, yet are quite effective. In particular, this approach allows for the consideration of semigroup theory within the numerical framework. As such, the method is able to exhibit favorable qualities.

Herein, we provide an approach where we approximate the true solution to the extended problem, as compared to the standard approach of discretizing the underlying system first. The approximation is shown to be stable, consistent, and convergent for appropriately smooth initial data. Moreoever, the employed Pad{\'e} approximation allows for the efficient construction of the numerical solution via factorization techniques. However, one could easily replace this approximation with any desired second-order (or higher) approximation to the associated abstract Cauchy problem.

Moving forward, we will consider improved approximations methods within the same abstract framework. In particular, we will consider abstract analysis of higher-order Pad{\'e} approximations and quadrature techniques, including Runge-Kutta and exponential integration methods, with the goal of constructing methods of arbitrary order. This will also allow for the consideration of the underlying algebraic structures of these methods, as has been done with Butcher and Lie-Butcher series for classical problems \cite{munthe2018lie}. We also intend to extend these results to abstract problems posed on abstract manifold structures, allowing for the consideration of novel problems in differential geometry. Moreover, we hope to better understand how well this extension technique preserves and approximates the Dirichlet-to-Neumann map that is of fundamental importance in the underlying theory. Preliminary work indicates that our proposed numerical method faithfully represents this operator after applying the conormal derivative given by $(\ref{e2})$, but there is still a need to finish developing a rigorous analysis of this application. However, more immediate concerns will be the deeper comparison of our newly proposed methods with popular existing methods for the fractional Laplacian problem.








\section*{Acknowledgements}

This work was supported by the NSF grant number 1903450. The author would like to thank Akif Ibraguimov, of Texas Tech University, for introducing this problem to them. The author is also thankful for Akif's time and insightful suggestions that greatly improved the quality of the work herein. Finally, the author is thankful to the reviewers who provided numerous useful suggestions that greatly improved the presentation of the numerical experiments in Section 5.

\bibliographystyle{plain}
\bibliography{Frac_Bib1}

\begin{thebibliography}{10}

\bibitem{abramowitz1965handbook}
M.~Abramowitz and I.~A. Stegun.
\newblock {\em Handbook of mathematical functions: with formulas, graphs, and
  mathematical tables}, volume~55.
\newblock Courier Corporation, 1965.

\bibitem{antil2015fem}
H.~Antil and E.~Ot{\'a}rola.
\newblock A {FEM} for an optimal control problem of fractional powers of
  elliptic operators.
\newblock {\em SIAM Journal on Control and Optimization}, 53(6):3432--3456,
  2015.

\bibitem{doi:10.1080/03605302.2017.1363229}
W.~Arendt, A.~F. M.~Ter Elst, and M.~Warma.
\newblock Fractional powers of sectorial operators via the
  {D}irichlet-to-{N}eumann operator.
\newblock {\em Communications in Partial Differential Equations}, 43(1):1--24,
  2018.

\bibitem{barbu2010nonlinear}
V.~Barbu.
\newblock {\em Nonlinear differential equations of monotone types in Banach
  spaces}.
\newblock Springer Science \& Business Media, 2010.

\bibitem{bonito2019numerical}
A.~Bonito, W.~Lei, and J.~E. Pasciak.
\newblock Numerical approximation of the integral fractional {L}aplacian.
\newblock {\em Numerische Mathematik}, 142(2):235--278, 2019.

\bibitem{budd1999geometric}
C.~J. Budd and A.~Iserles.
\newblock Geometric integration: numerical solution of differential equations
  on manifolds.
\newblock {\em Philosophical Transactions of the Royal Society of London A:
  Mathematical, Physical and Engineering Sciences}, 357(1754):945--956, 1999.

\bibitem{doi:10.1080/03605300600987306}
L.~Caffarelli and L.~Silvestre.
\newblock An extension problem related to the fractional {L}aplacian.
\newblock {\em Communications in Partial Differential Equations},
  32(8):1245--1260, 2007.

\bibitem{10.1086/338705}
P.~Carr, H.~Geman, D.~B. Madan, and M.~Yor.
\newblock The fine structure of asset returns: An empirical investigation.
\newblock {\em The Journal of Business}, 75(2):305--332, 2002.

\bibitem{chen2015pde}
L.~Chen, R.~H. Nochetto, E.~Ot{\'a}rola, and A.~J. Salgado.
\newblock A {PDE} approach to fractional diffusion: a posteriori error
  analysis.
\newblock {\em Journal of Computational Physics}, 293:339--358, 2015.

\bibitem{cushman1993nonlocal}
J.~H. Cushman and T.~R. Ginn.
\newblock Nonlocal dispersion in media with continuously evolving scales of
  heterogeneity.
\newblock {\em Transport in Porous Media}, 13(1):123--138, 1993.

\bibitem{deimling2010nonlinear}
K.~Deimling.
\newblock {\em Nonlinear functional analysis}.
\newblock Courier Corporation, 2010.

\bibitem{d2013fractional}
M.~{D`Elia} and M.~Gunzburger.
\newblock The fractional {L}aplacian operator on bounded domains as a special
  case of the nonlocal diffusion operator.
\newblock {\em Computers \& Mathematics with Applications}, 66(7):1245--1260,
  2013.

\bibitem{eringen2002nonlocal}
A.~C. Eringen.
\newblock {\em Nonlocal continuum field theories}.
\newblock Springer Science \& Business Media, 2002.

\bibitem{Gale2013}
J.~E. Gal{\'e}, P.~J. Miana, and P.~R. Stinga.
\newblock Extension problem and fractional operators: semigroups and wave
  equations.
\newblock {\em Journal of Evolution Equations}, 13(2):343--368, Jun 2013.

\bibitem{hairer2006geometric}
E.~Hairer, C.~Lubich, and G.~Wanner.
\newblock {\em Geometric numerical integration: structure-preserving algorithms
  for ordinary differential equations}, volume~31.
\newblock Springer Science \& Business Media, 2006.

\bibitem{hansen2013convergence}
E.~Hansen and E.~Henningsson.
\newblock A convergence analysis of the {P}eaceman--{R}achford scheme for
  semilinear evolution equations.
\newblock {\em SIAM Journal on Numerical Analysis}, 51(4):1900--1910, 2013.

\bibitem{henry2006geometric}
D.~Henry.
\newblock {\em Geometric theory of semilinear parabolic equations}, volume 840.
\newblock Springer, 2006.

\bibitem{huang2014numerical}
Y.~Huang and A.~Oberman.
\newblock Numerical methods for the fractional {L}aplacian: A finite
  difference-quadrature approach.
\newblock {\em SIAM Journal on Numerical Analysis}, 52(6):3056--3084, 2014.

\bibitem{Iserles:2008:FCN:1506265}
Arieh Iserles.
\newblock {\em A First Course in the Numerical Analysis of Differential
  Equations}.
\newblock Cambridge University Press, New York, NY, USA, 2nd edition, 2008.

\bibitem{JONES2016207}
T.~Jones, L.~P. Gonzalez, S.~Guha, and Q.~Sheng.
\newblock A continuing exploration of a decomposed compact method for highly
  oscillatory wave problems.
\newblock {\em Journal of Computational and Applied Mathematics}, 299:207 --
  220, 2016.
\newblock Recent Advances in Numerical Methods for Systems of Partial
  Differential Equations.

\bibitem{lumer1961dissipative}
G.~Lumer and R.~S. Phillips.
\newblock Dissipative operators in a {B}anach space.
\newblock {\em Pacific Journal of Mathematics}, 11(2):679--698, 1961.

\bibitem{martinez2001theory}
C.~Martinez and M.~Sanz.
\newblock {\em The theory of fractional powers of operators}, volume 187.
\newblock Elsevier, 2001.

\bibitem{meichsner2017fractional}
J.~Meichsner and C.~Seifert.
\newblock Fractional powers of non-negative operators in {B}anach spaces via
  the {D}irichlet-to-{N}eumann operator.
\newblock {\em arXiv preprint arXiv:1704.01876}, 2017.

\bibitem{munthe2018lie}
H.~Z. Munthe-Kaas and K.~K. F{\o}llesdal.
\newblock Lie--{B}utcher series, geometry, algebra and computation.
\newblock In {\em Discrete Mechanics, Geometric Integration and Lie--Butcher
  Series}, pages 71--113. Springer, 2018.

\bibitem{nochetto2015pde}
R.~H. Nochetto, E.~Ot{\'a}rola, and A.~J. Salgado.
\newblock A {PDE} approach to fractional diffusion in general domains: a priori
  error analysis.
\newblock {\em Foundations of Computational Mathematics}, 15(3):733--791, 2015.

\bibitem{nochetto2016pde}
R.~H. Nochetto, E.~Otarola, and A.~J. Salgado.
\newblock A {PDE} approach to space-time fractional parabolic problems.
\newblock {\em SIAM Journal on Numerical Analysis}, 54(2):848--873, 2016.

\bibitem{Josh1}
J.~L. Padgett and Q.~Sheng.
\newblock On the positivity, monotonicity, and stability of a semi-adaptive
  {LOD} method for solving three-dimensional degenerate {K}awarada equations.
\newblock {\em J. Math. Anal. Appls}, 439:465--480, 2016.

\bibitem{PADGETT2018210}
J.~L. Padgett and Q.~Sheng.
\newblock Numerical solution of degenerate stochastic kawarada equations via a
  semi-discretized approach.
\newblock {\em Applied Mathematics and Computation}, 325:210 -- 226, 2018.

\bibitem{silling2000reformulation}
S.~A. Silling.
\newblock Reformulation of elasticity theory for discontinuities and long-range
  forces.
\newblock {\em Journal of the Mechanics and Physics of Solids}, 48(1):175--209,
  2000.

\bibitem{soderlind2006logarithmic}
G.~S{\"o}derlind.
\newblock The logarithmic norm. history and modern theory.
\newblock {\em BIT Numerical Mathematics}, 46(3):631--652, 2006.

\bibitem{doi:10.1080/03605301003735680}
P.~R. Stinga and J.~L. Torrea.
\newblock Extension problem and {H}arnack's inequality for some fractional
  operators.
\newblock {\em Communications in Partial Differential Equations},
  35(11):2092--2122, 2010.

\end{thebibliography}





\end{document}